\documentclass{article}
\usepackage{amssymb,amsmath,amsthm,mathrsfs} 
\usepackage{verbatim}
\usepackage{tikz-cd} 
\usepackage{hyperref}
\usepackage{tabularx}
\usepackage{pgfplots}
\usepackage{pgfplotstable}
\usepackage{filecontents}
\usepackage{geometry}
\usepackage{mathtools}
\usepackage{float}
\usepackage{booktabs}
\usepackage{lipsum}

\let\phi\varphi

\newtheorem{theorem}{Theorem}[section]
\newtheorem{corollary}[theorem]{Corollary}
\newtheorem{lemma}[theorem]{Lemma}
\newtheorem{proposition}[theorem]{Proposition}

\theoremstyle{definition}

\newtheorem{remark}[theorem]{Remark}

\newtheorem{example}[theorem]{Example}

\newcommand*{\bbb}[1]{{\mathbb{#1}}}
\newcommand{\C}{\bbb{C}}

\newcommand{\Q}{\bbb{Q}}

\newcommand{\Z}{\bbb{Z}}
\newcommand{\F}{\bbb{F}}

\newcommand{\Gal}{\text{Gal}}

\newcommand{\Frob}{\text{Frob}}

\newcommand{\End}{\text{End}}
\newcommand{\Mod}[1]{\ (\mathrm{mod}\ #1)}

\newcommand{\p}{\mathfrak{p}}

\newcommand{\mO}{\mathcal{O}}

\title{On CM Elliptic Curves and the Cyclotomic $\lambda$-Invariants of Imaginary Quadratic Fields}

\author{Matt Stokes}

\begin{document}

\pagestyle{plain}
\maketitle

\begin{abstract}
Let $K$ be an imaginary quadratic field, and fix a prime $p > 3$ that does not divide the class number of $K$.  In this paper we prove that Iwasawa's $\lambda$-invariant for the cyclotomic $\Z_p$-extension of $K$ is greater than $1$ if and only if the number of points on a certain reduced elliptic curve is divisible by $p^2$.
\end{abstract}

\section{Introduction}

In this paper we will prove a surprising connection between the classical Iwasawa $\lambda$-invariants for the cyclotomic $\Z_p$-extension of an imaginary quadratic field $K$ (these are the $\lambda$-invariants that track the growth of the Sylow-$p$ subgroups of the class groups associated to the $\Z_p$-extension), and certain $p$-powered torsion of a reduced elliptic curve with complex multiplication by $\mO_K$.  In this introduction we will first give a brief review of Iwasawa's growth formula (Theorem \ref{igf}), and then give an overview of classical Iwasawa theory for imaginary quadratic fields.   We will then state the main result of this paper.

\subsection{The Growth of Class Groups in a $\Z_p$-extension}

Given a tower of number fields $F_0 \subset F_1 \subset F_2 \subset \dots$, one may wish to study how the class group $Cl(F_n)$ behaves as $n$ increases.  The general behavior of these groups can be quite chaotic.  However, Iwasawa theory gives us a setting where much more can be said.

Let $F$ be a number field, and $p$ a prime.  Suppose that $\mathcal{F}/F$ is Galois such that $\Gal(\mathcal{F}/F) \cong \Z_p$.  Then we say that $\mathcal{F}/F$ is a $\Z_p$-extension of $F$.  Such an extension always exists, and can be seen as follows:  For $n \in \Z^+$, denote $\zeta_{p^n}$ to be a $p^n$-th root of unity.  Without loss of generality, we will assume that $\zeta_p \not\in F$.  Then the extension $F(\zeta_{p^{n+1}})/K$ is Galois with Galois group isomorphic to $\Z/p^{n}(p-1)\Z$, and so contains a sub-extension $F_n/F$ such that $\Gal(F_n/F) \cong \Z/p^{n}\Z$.  Taking $F_{\infty} = \bigcup F_n$, we have that $\Gal(F_{\infty}/F) = \varprojlim \Z/p^n\Z \cong \Z_p$.  By virtue of this construction, we say that $F_{\infty}/F$ is the cyclotomic $\Z_p$-extension of $F$.  For a general $\Z_p$-extension $\mathcal{F}/F$, we will have for each $n \in \Z^+$ a Galois extension $F_n / F$ such that $\Gal(F_n/F) \cong \Z/p^n\Z$, $F_n \subset F_{n+1}$, and $\bigcup F_n = \mathcal{F}$.  

Now, given a $\Z_p$-extension $\mathcal{F}/F$, we denote $A_n = A(F_n)$ to be the Sylow-$p$ subgroup of $Cl(F_n)$.  Then 
\begin{theorem}[Iwasawa]\label{igf}
With $\mathcal{F}/F$ as above, there exists $\lambda, \mu \in \Z^+$ and $\nu \in \Z$ such that for $n$ sufficiently large,
\[
\#A_n = p^{n \lambda + p^n \mu + \nu}.
\]
\end{theorem}

We say that $\lambda, \mu$ and $\nu$ are the Iwasawa invariants of $\mathcal{F}/F$.  At times our notation for the Iwasawa invariants may indicate the prime $p$, the field $F$, or the $\Z_p$-extension $\mathcal{F}/F$ (e.g. $\lambda_p(\mathcal{F}/F)$).  It was conjectured by Iwasawa that the $\mu = 0$ for the Cyclotomic $\Z_p$-extension of $F$.  This is still an open problem, however, Ferrero and Washington \cite{fw} have shown that the cyclotomic $\mu$-invariant vanishes when $F$ is an Abelian extension of $\Q$.  Washington has also shown in \cite{wash2} (see also Theorem 16.12 on page 385 of \cite{wash}) that when $F/\Q$ is Abelian and $\mathcal{F}/F$ is the cyclotomic $\Z_p$-extension of $F$, the non-$p$-part of $Cl(F_n)$ is bounded as $n$ goes to infinity.

\subsection{Cyclotomic $\Z_p$-extensions of Imaginary Quadratic Fields} 

Here we aim to convince the reader that imaginary quadratic fields serve as the first non-trivial case in which to study Iwasawa theory.  We also state some important results in this area.  For a number field $F$, denote $r$ to be the number of pairs of complex embeddings of $F$, and $\tilde{F}$ to be the compositum of all $\Z_p$-extensions of $F$.  If $F/\Q$ is Abelian, then we have $\Gal(\tilde{F}/F) \cong \Z_p^{r +1}$, which tells us that any totally real Abelian number field has only one $\Z_p$-extension (the cyclotomic one).  On the other hand, if $K$ is an imaginary quadratic field, we have that $\Gal(\tilde{K}/K) \cong \Z_p^2$, so that there are infinitely many $\Z_p$-extensions of $K$.  For more on the general case (i.e. $F/\Q$ non-Abelian) see Theorem 13.4 in Chapter 13 of \cite{wash}.  The following result of Iwasawa tells us everything we need to know about the $\Z_p$-extension of $\Q$:

\begin{theorem}[Iwasawa \cite{iwasawa0}]

Suppose that $L/F$ is a Galois extension of number fields such that $\Gal(L/K) \cong \Z/p^n\Z$ for some $n \in \Z^+$, and that exactly one prime ramifies in $L/F$ (this prime does not have to lie above $p$).  Then $p$ divides $\#Cl(L)$ only if $p$ divides $\#Cl(F)$.
\end{theorem}

From this we see that the cyclotomic invariants for $\Q$ are zero.  Indeed, given a layer $\Q_n$ of the cyclotomic $\Z_p$-extension of $\Q$, we have that $p$ is the only prime that ramifies in $\Q_n/\Q$ (Proposition 13.2 in \cite{wash}), and since $\#Cl(\Q) = 1$ it must be that $\#A(\Q_n) = 1$.  The triviality of the Iwasawa invariants of $\Q$ is an instance of Greenberg's conjecture, which predicts $\lambda = \mu = 0$ for the cyclotomic $\Z_p$-extension of a totally real number field.  In particular, the cyclotomic $\lambda$ invariant is believed to always be zero for real quadratic fields.  All of this together then leads one to study the $\Z_p$-extensions of imaginary quadratic fields.

Let $K$ be an imaginary quadratic field, and $p$ a prime.  If $\mathcal{K}/K$ is a $\Z_p$-extension of $K$, we will denote $\lambda_p(\mathcal{K}/K)$ to be the corresponding $\lambda$-invariant.  We will also denote $\lambda_p(K)$ to be the lambda invariant for the cyclotomic $\Z_p$-extension of $K$.  We have the following well known facts about imaginary quadratic fields (see, for example, Propositions 4.3 and 4.4 in \cite{sands1}):
\begin{proposition}
Suppose that $\mathcal{K}/K$ is a $\Z_p$-extension. Then we have:
\begin{itemize}
\item[1.]If $p$ is a prime that does not split in $K$, and does not divide $\#Cl(K)$, then $\mu_p(\mathcal{K}/K) = \lambda_p(\mathcal{K}/K) = 0$. 
\item[2.] If $p$ splits in $K$, and the divisors of $p$ in $K$ are totally ramified in $\mathcal{K}/K$, the $\lambda_p(\mathcal{K}/K) > 0$.
\end{itemize}
\end{proposition}

Next, we have the powerful criterion of Gold, which will be a main ingredient in the proof of our main result:

\begin{theorem}[Gold's criterion \cite{gold}] \label{gold}
Let $K$ be an imaginary quadratic field, and $p > 3$ be a prime such that $p$ does not divide the class number $h_K$ of $K$.  Suppose $p\mO_K = \frak{p}\bar{\frak{p}}$ and write $\frak{p}^{h_K} = (\alpha)$. Then $\lambda_p(K) > 1$ if and only if $\alpha^{p-1} \equiv 1 \Mod{\bar{\frak{p}}^2}$.
\end{theorem}

For a fixed imaginary quadratic field $K$, Theorem \ref{gold} can be used to efficiently detect primes $p$ for which $\lambda_p(K) >1$, as carried out by Dummit, Ford, Kisilevsky and Sands in \cite{dfks}.  Their computations suggest that the cyclotomic $\lambda$-invariants for an imaginary quadratic field are rarely greater than one (see the Table 1 at the end of \cite{dfks}).   In another application, for a fixed prime $p$, Sands applies Theorem \ref{gold} in \cite{sands1} to construct infinitely many imaginary quadratic fields $K$ such that $\lambda_p(K) > 1$.  There is a sort of dual statement to Sands result that one may wonder about:  Given a fixed imaginary quadratic field $K$, are there infinitely many primes $p$ with $\lambda_p(K) > 1$?  This was conjectured to be true in \cite{dfks}, but is currently an open question.

There is also the following result of Sands which shows that the cyclotomic $\Z_p$-extension may tell us a lot about the other $\Z_p$-extensions of $K$:
\begin{theorem}[Sands \cite{sands1}]
Suppose that $p$ does not divide $h_K$, and $\lambda_p(K) \leq 2$.  Then each $\Z_p$-extension $\mathcal{K}/K$  has $\mu_p(\mathcal{K}/K) = 0$.  We also have
\begin{itemize}
\item[a.] If $\lambda_p(K) = 0$, then $\lambda_p(\mathcal{K}/K) = 0$ for each $\Z_p$-extension $\mathcal{K}/K$.
\item[b.] If $\lambda_p(K) = 1$, then $\lambda_p(\mathcal{K}/K) = 1$ for each $\Z_p$-extension $\mathcal{K}/K$ in which each prime dividing $p$ is totally ramified.
\item[c.] If $\lambda_p(K) = 2$, then $1\leq \lambda_p(\mathcal{K}/K) \leq 2$ for each $\Z_p$-extension $\mathcal{K}/K$ in which each prime dividing $p$ is totally ramified.
\end{itemize}
\end{theorem}

\subsection{Statement of Main Result}
Let $K$ be an imaginary quadratic field, $p$ a prime such that $p$ splits in $K$ as $p\mO_K = \frak{p}\bar{\frak{p}}$, and denote $\lambda_p(K)$ to be Iwasawa's $\lambda$-invariant for the cyclotomic $\Z_p$-extension of $K$.  Recall that in this situation $\lambda_p(K) > 0$.  We will consider an elliptic curve $E$ that has complex multiplication by $\mO_K$, i.e. $\End(E) \cong \mO_K$.  Let $j(E)$ be the $j$-invariant of $E$, and $\mathfrak{P}$ be a prime of $K(j(E))$ (the Hilbert class field of $K$) above $\mathfrak{p}$ for which $E$ has good reduction.  Denote $\tilde{E}$ to be the reduction of $E$ modulo $\mathfrak{P}$.  Here is the main result of this paper, which we will prove in Section \ref{proof}:

\begin{theorem}\label{thrm1}
With the notation fixed above, we have $\lambda_p(K) > 1$ if and only if $\# \tilde{E}(\F_{q}) \equiv 0 \Mod{p^2}$, where $q = p^{p-1}$.
\end{theorem}

The idea of Theorem \ref{thrm1} comes from the proof of Gold's criterion (\cite{gold}, Theorem 4).  In that proof, Gold argues that $\lambda_p(K) > 1$ if and only if $\frak{p}$ splits in a certain ray-class field modulo $\bar{\frak{p}}^2$ over $K$.  We may explicitly realize such fields via the theory of complex multiplication, and we then draw out the implications for the curve $E$ when $\frak{p}$ splits in this ray-class field extension of $K$.

In Section \ref{applications} we give a few applications of Theorem \ref{thrm1}.  First, we obtain some interesting corollaries for the curves $y^2 = x^3 -1$ and $y^2 = x^3 + x$ (see Propositions \ref{thrm2}, \ref{thrm3} and Corollary \ref{cor1}).  We then give another condition for $\lambda_p(K) > 1$, this time involving a character sum whose terms include the quadratic character $\psi_q : \F_q \to \{\pm 1\}$ (see Proposition \ref{prop2}).  Last, given a fixed prime $p$, we apply Theorem \ref{thrm1} to identify a family of imaginary quadratic fields with $\lambda_p = 1$ (Proposition \ref{prop1}).

\section{A Few Examples}

Before we get to the proof of Theorem \ref{thrm1}, let us first check some computations.  Given a prime $p>3$, we will denote $q = p^{p-1}$.  In the following example, each field $K$ has class number $1$, so the corresponding elliptic curves with complex multiplication by $\mO_K$ are defined over $\Q$.  Therefore, we may reduce each curve by the rational prime $p$.
\vskip5mm
\begin{tabular}{m{15em} m{15em} m{15em}}
$K = \Q(\sqrt{-3})$  & $K = \Q(\sqrt{-11})$ & $K = \Q(\sqrt{-19})$
 \\
$E: y^2 = x^3 -1$ & $E: y^2 = x^3 -264x -1694$ & $E: y^2 = x^3 -608x -5776$ \\
\begin{tabular}{ |m{5em}|m{5em}|} 
 \hline
$3< p < 70$ such that $p$ splits in $K$ & $\#\tilde{E}(\F_q)$  (mod $p^2$) \\ 
 \hline
 7 & 42  \\ 
 \hline
 13 & 0  \\ 
 \hline
 19 & 342  \\ 
 \hline
 31 & 527 \\ 
 \hline
 37 & 1332 \\
\hline
 43 & 559  \\ 
 \hline
 61 & 3660  \\ 
\hline
 67 & 3685  \\ 
 \hline
\end{tabular}
&
\begin{tabular}{ |m{5em}|m{5em}|} 
 \hline
$3< p < 70$ such that $p$ splits in $K$ & $\#\tilde{E}(\F_q)$  (mod $p^2$) \\ 
 \hline
 5 & 0  \\ 
 \hline
 23 & 230  \\ 
 \hline
 31 & 527  \\ 
 \hline
 47 & 1222  \\ 
\hline
 53 & 1007  \\ 
 \hline
 59 & 59  \\ 
\hline
 67 & 1340  \\ 
 \hline
\end{tabular}
&
\begin{tabular}{ |m{5em}|m{5em}|} 
 \hline
$3< p < 70$ such that $p$ splits in $K$ & $\#\tilde{E}(\F_q)$  (mod $p^2$) \\ 
 \hline
 5 & 20  \\ 
 \hline
 7 & 28  \\ 
 \hline
 11 & 0  \\ 
 \hline
 17 & 102  \\ 
\hline
 23 & 506  \\ 
 \hline
 43 & 1806  \\ 
\hline
 47 & 893  \\ 
 \hline
 61 & 3355 \\
 \hline
\end{tabular}
\end{tabular}
\vskip5mm
Consulting Table 1 in \cite{dfks}, we see for the primes $p < 70$, and the fields $K = \Q(\sqrt{-3}), \Q(\sqrt{-11})$, and $\Q(\sqrt{-19})$, that $\lambda_p(K) > 1$ precisely when the number of reduced points of the respective curves is zero modulo $p^2$.

\section{Some Preliminaries}
The main idea for the proof of Theorem \ref{thrm1} is that $\lambda_p(K) > 1$ if and only if $\frak{p}$ splits in the $p$-ray class field of $K$ modulo $\bar{\frak{p}}^2$, that is, the fixed field of the Sylow-$p$ subgroup of the ray class group of $K$ modulo $\bar{\frak{p}}^2$ (see the proof of Theorem 4 in \cite{gold}).  Via the theory of complex multiplication, this $p$-ray class field of $K$ is essentially obtained by adding coordinates of certain torsion points of an elliptic curve $E$ with complex multiplication by $\mO_K$.  In this section, we will first review some facts about the theory of complex multiplication.  Then we will go over some preliminary results needed for the proof of Theorem \ref{thrm1}.  For more on the theory of complex multiplication, see Chapter II in \cite{silverman1}.

\subsection{Complex Multiplication and Abelian Extensions of Imaginary Quadratic Fields}
For any elliptic curve $E$ and $m \in \Z$, we let $E[m] = \{P \in E(\bar{K}) \,:\, [m]P = O\}$ denote the $m$-torsion points of $E$ (as usual, $\bar{K}$ is the algebraic closure of $K$).  From the theory of complex multiplication, we can obtain Abelian extensions of $K$ by essentially adjoining the torsion points of $E$.  More precisely, if $F/K$ is an Abelian extension, then there exists some $m \in \Z^+$ such that
\[
F \subseteq K(j(E),E[m]) = K(j(E), \{ x,y \in \bar{K} \, : \, (x,y) \in E[m]\})
\]
where $j(E)$ is the $j$-invariant of $E$.  Note that this is analogous to the Kronecker-Weber Theorem:
\begin{theorem}[Kronecker-Weber]
Suppose that $F$ is an Abelian extension of $\Q$.  Then there exists $n \in \Z^+$ such that $F \subseteq \Q(\zeta_n)$, where $\zeta_n$ is a primitive $n$-th root of unity.
\end{theorem}   
Something interesting happens when we adjoin the $x$-coordinates of the torsion points of $E$.  Consider the Webber function $\Phi: E[m] \to \bar{K}$ given by
\begin{align*}
\Phi(x,y) = 
\begin{cases}
x^2  & \text{ if $K = \Q(i)$}
\\
x^3  & \text{ if $K = \Q(\sqrt{-3})$}
\\
x  & \text{else}
\end{cases}
\end{align*}
Since $E$ has complex multiplication by $\mO_K$, we may also define for a fractional ideal $\mathfrak{c}$ of $K$, the $\frak{c}$-torsion points of $E$,
\[
E[\frak{c}] = \{ P \in E(\bar{K}) \,:\, [\alpha] P = O \text{ for all } \alpha \in \frak{c}\},
\]
and the Abelian extension $K(j(E),\Phi(E[\frak{c}]))/K$, where
\[
K(j(E), \Phi(E[\frak{c}])) = K(j(E),\{ \Phi(P) \, : \, P \in E[\frak{c}]\}).
\]
Then we have that $K(j(E), \Phi(E[\frak{c}]))$ is the ray class field of $K$ modulo $\frak{c}$.

\subsection{Preliminary Results}

Let $K$ be an imaginary quadratic field.  Throughout we will assume that $E/K(j(E))$ is an elliptic curve with complex multiplication by $\mO_K$, and that $p\mO_K = \frak{p}\bar{\frak{p}}$.  Denote $H$ to be the $p$-ray class field of $K$, and $\widetilde{K} = K(j(E), E[\bar{\frak{p}}^2])$.  Let be a prime of $\widetilde{K}$ above $\frak{p}$.  As mentioned before, we can relate the $\lambda$-invariant of $K$ and certain $p$-power torsion of a reduced elliptic curve with complex multiplication by $\mO_K$ via the theory of complex multiplication.  This is done by restating a result of Gold (\cite{gold}, Theorem 4):

\begin{proposition}[Gold]\label{prop1}
Let $H$ be the $p$-ray class field modulo $\bar{\frak{p}}^2$, and $\sigma_{\frak{p}} \in \Gal(H/K)$ the image of $\frak{p}$ under the Artin map.  Then $\lambda_{p}(K) > 1$ if and only if $\sigma_{\frak{p}} = 1$.
\end{proposition}

\begin{proof}
We have that $\lambda_p(K) > 1 $ if and only if $\frak{p}$ splits in $H/K$ (see \cite{gold}, proof of Theorem 4, and Theorem 3 in \cite{gold}).  Since $[H:K] = p$, the prime $\frak{p}$ splits in $H/K$ if and only if it splits completely in $H/K$.  The Proposition now follows.
\end{proof}

We now seek to understand how certain $p$-power torsion points of $E$ reduce modulo $\frak{P}$.  To that end, we have the following:

\begin{theorem}[Deuring's reduction criterion \cite{deuring}]\label{deuring}
Let $F/\Q$ be a finite extension, $E/F$ an elliptic curve with complex multiplication by $\mO_K$, and $\mathfrak{P}$ a prime of $F$ above $p$ for which $E$ has good reduction.  Then $E$ has ordinary reduction at $\mathfrak{P}$ if and only if $p$ splits in $K$.
\end{theorem}

\begin{lemma}\label{lem1}
Let $E/K(j(E))$ be an elliptic curve with complex multiplication by $\mO_K$, and let $r \in \Z^+$.  If $\frak{P}$ is a prime of $\widetilde{K}$ above $\frak{p}$ such that $E$ has good reduction at $\frak{P}$, then,
\begin{enumerate}
\item[a.] $E[p^r] \cong E[\frak{p}^r] \oplus E[\bar{\frak{p}}^r]$.
\item[b.] The reduction modulo $\frak{P}$ map $E[\bar{\frak{p}}^r] \to \tilde{E}[p^r]$ is an isomorphism.
\end{enumerate}
\end{lemma}

\begin{proof}
Proof of (a):  This will follow from the more general fact if $I$ and $J$ are coprime ideals of $\mO_K$, then $E[IJ] = E[I] \oplus E[J]$.  Indeed, since $I$ and $J$ are coprime, we can find $a \in I$ and $b \in J$ such that $a + b = 1$, and for any $P \in E[I] \cap E[J]$, we have $P = [a + b]P = [a]P + [b]P = O$.  Now, clearly $E[I]$ and $E[J]$ are normal subgroups of $E[IJ]$, and both $E[IJ]$ and $E[I] \oplus E[J]$ have order $N_{K/\Q}(I) \cdot N_{K/\Q}(J)$.  Since $\frak{p}$ and $\bar{\frak{p}}$ are coprime, this completes the proof of part (a).

Proof of (b):  Before the proof, recall that if $L/F$ is Galois with $\frak{Q}$ a prime of $L$ above a prime $\frak{q}$ of $F$, then we have the decomposition subgroup with respect to $\frak{Q}/\frak{q}$
\[
Z(\frak{Q}/\frak{q}) = \{ \sigma \in \Gal(L/F) \, : \, \sigma\frak{Q} = \frak{Q} \}
\]
and the inertia group with respect to $\frak{Q}/\frak{q}$
\[
T( \frak{Q}/ \frak{q}) = \{ \sigma \in Z(\frak{Q}/\frak{q}) \, : \, \sigma(x) \equiv x \Mod{\frak{Q}} \text{ for all $x \in L$}\}.
\]
If $G(\frak{Q}/\frak{q}) = \Gal(\mO_L/\frak{Q}/\mO_F/\frak{q})$, then we have a map $Z(\frak{Q}/\frak{q}) \to G(\frak{Q}/\frak{q})$, given by
\[
\sigma \mapsto \bar{\sigma}: x + \frak{Q} \mapsto \sigma(x) + \frak{Q}.
\]
Now, back to the proof.  Let $g_{\frak{p}}$ be the order of the class of $\frak{p}$ in $Cl(K)$, and consider $\widetilde{L} = K(j(E), E[p^{g_{\p}}])$, and $L = K(j(E))$.  The field $\widetilde{K} \subseteq \widetilde{L}$ will be as it was defined above.  Let $\mathcal{P}$ be a prime of $L$ above $\frak{p}$, and $\mathscr{P}$ be a prime of $\widetilde{L}$ above $\mathcal{P}$.  Let $\tau_{\frak{p}} \in \Gal(L/K)$ be the image of $\frak{p}$ under the Artin map, and $\tau \in \Gal(\widetilde{L}/K)$ such that $\tau$ maps to the Frobenius automorphism in $G(\mathscr{P}/\frak{p})$.  Thus, $\tau|_L = \tau_{\frak{p}}$, and $E^{\tau^{g_{\p}}} = F(\tau^{g_{\p}})*E = \frak{p}^{g_{\p}}*E$.  Then $\tau^{g_{\p}}$ induces an isogeny $\eta : E \to E$ that gives rise to the commutative diagram
\[ 
\begin{tikzcd}
E \arrow{r}{[\alpha]} \arrow[swap]{d}{} & E \arrow{d}{} \\%
\tilde{E} \arrow{r}{\text{$\phi$}}& \tilde{E}
\end{tikzcd}
\]
where $\tilde{E}$ is the reduction of $E(\widetilde{L})$ modulo $\mathscr{P}$, the map $\phi$ is the $p^{g_{\p}}$-th power Frobenius, and $\frak{p}^{g_{\p}} = (\alpha)$ (see Theorem 5.3 and Corollary 5.4 in \cite{silverman1}).  Since $E[\frak{p}^{g_{\p}}] = \ker [\alpha]$, we have that $\tilde{E}[\frak{p}^{g_{\p}}] \subseteq \ker \phi$, so that $\tilde{E}[\frak{p}^{g_{\p}}]$ is trivial.  But by Deuring's reduction criterion \ref{deuring}, we have that $\tilde{E}[p^{g_{\p}}] \cong \Z/p^{g_{\p}}\Z$.  So it must be that $\tilde{E}[p^{g_{\p}}] = \tilde{E}[\bar{\frak{p}}^{g_{\p}}]$.  Now, for a positive integer $r < g_{\p}$, define $\widetilde{K}_r = K(j(E), E[\bar{\frak{p}}^r]) \subseteq \widetilde{L}$, and let $\frak{P}_r = \mathscr{P} \cap \widetilde{K}_r$.  Notice we can view any $P  \in E(\widetilde{K}_r)$ an element of $E(\widetilde{L})$.  Therefore, the reduction of $P$ modulo $\mathscr{P}$ is the same as reduction modulo $\frak{P}_r$.  We now have the commutative diagram
\[ 
\begin{tikzcd}
E(\widetilde{K}_r) \arrow{r}{} \arrow[swap]{d}{\text{$\Mod{\frak{P}_r}$}} & E(\widetilde{L}) \arrow{d}{\text{$\Mod{\mathscr{P}}$}} \\%
\tilde{E}(\mO_{\widetilde{K}_r}/\frak{P}_r) \arrow{r}{}& \tilde{E}(\mO_{\widetilde{L}}/\mathscr{P})
\end{tikzcd}
\]
where the horizontal maps are inclusion, and the vertical maps are reduction modulo $\frak{P}_r$ and $\mathscr{P}$ respectively.  From this diagram, we see that $E[\bar{\p}^r] \subseteq E[\bar{\p}^{g_{\p}}]$ is non-trivial modulo $\mathscr{P}$, so $E[\bar{\p}^r]$ is non-trivial modulo $\frak{P}_r$.
\end{proof}


\section{Proof of Theorem \ref{thrm1}}\label{proof}

\begin{proof}[Proof of Theorem \ref{thrm1}]
As before, let $H$ be the $p$-ray class group of $K$ modulo $\frak{p}^2$, denote $\widetilde{K} = K(j(E),E[\bar{\frak{p}}^2])$, $\mathscr{P}$ a prime of $H$ above $\frak{p}$, and $\frak{P}$ a prime of $\widetilde{K}$ above $\mathscr{P}$.  We have the following commutative diagram,
\[
\begin{tikzcd}
1 \arrow{r}{} &  T(\frak{P}/\frak{p}) \arrow{r}{} \arrow[swap]{d}{} & Z(\frak{P}/\frak{p}) \arrow{r}{}  \arrow[swap]{d}{} & G(\frak{P}/\frak{p})  \arrow{r}{} \arrow[swap]{d}{} & 1  \\%
1 \arrow{r} & T(\mathscr{P}/\frak{p}) \arrow{r}{} & Z(\mathscr{P}/\frak{p}) \arrow{r}{} & G(\mathscr{P}/\frak{p})  \arrow{r}{} & 1
\end{tikzcd}
\]
where the vertical maps are restriction from $\widetilde{K}$ to $H$.  Let $\tau \in Z(\frak{P}/\p)$ be a lift of $\Frob(\frak{P}/\frak{p}) \in G(\frak{P}/\frak{p})$, and $\bar{\tau} \in \Gal(K(E[\bar{\p}^2])/K)$ the restriction of $\tau$ to $\widetilde{H} = K(E[\bar{\p}^2])$.  Then $\bar{\tau}|_H = \sigma_{\p}$.  From Lemma \ref{lem1} we have another commutative diagram
\[ 
\begin{tikzcd}
E[\bar{\frak{p}}^2] \arrow{r}{\bar{\tau}} \arrow[swap]{d}{\cong} & E[\bar{\frak{p}}^2]\arrow{d}{\cong} \\%
\tilde{E}[p^2] \arrow{r}{\text{$\phi$}}& \tilde{E}[p^2]
\end{tikzcd}
\]
where $\phi$ is the $p$-power Frobenius automorphism.  Let $\widetilde{\frak{P}} = \frak{P} \cap \mO_{\widetilde{H}}$.  Then the reduction of $E[\bar{\p}^2]$ modulo $\frak{P}$ is the same as reduction modulo $\widetilde{\frak{P}}$.  If $f = [\mO_{\widetilde{H}}/\widetilde{\frak{P}}:\mO_K/\frak{p}] > 1$, then for any positive integer $a < f$, we have that $\bar{\tau}^a$ will act non-trivially on $E[\bar{\frak{p}}^2]$.  

Assume that $\lambda_p(K) > 1$.  Then $\sigma_{\frak{p}} \in \Gal(H/K)$ is trivial  by Proposition \ref{prop1}, and from the proof of Theorem 2.3 in Chapter II of \cite{silverman1}, we have that $[\widetilde{H} : K]$ divides $p(p-1)$.  Thus, $f$ divides $p-1$ so that $\bar{\tau}^{p-1}$ acts trivially on $E[\bar{\frak{p}}^2]$.  From the commutative diagram, we then have that the $p^{p-1}$-power Frobenius fixes $\tilde{E}[p^2]$, which gives the forward implication.  

Now assume that $\sigma_{\p}$ is non-trivial in $\Gal(H/K)$ (i.e. $\lambda_p(K) \leq 1$).  Then, $\frak{p}$ is unramified and does not split in $H/K$.  Hence, the residue degree of $\frak{p}$ in $H/K$ is $p$, and $f \geq p > p-1$.  So, $\bar{\tau}^{p-1}$ does not act trivially on $E[\bar{\frak{p}}^2]$, hence the $p^{p-1}$-power Frobenius does not act trivially on $\tilde{E}[p^2]$.  This gives the reverse implication.  
\end{proof}

\section{Some Applications}\label{applications}

Suppose that $p$ is a prime and $m \in \Z^+$ such that $p \equiv 1 \Mod{m}$.  We say that $p$ is $1$-exceptional for $m$ if
\[
\left(\frac{p^2 -1}{m}\right)^{p-1}_p! = \prod_{\mathclap{a = 1 \atop \gcd(a,p) = 1}}^{\frac{p^2-1}{m}} a^{p-1} \equiv 1 \Mod{p^2}.
\]  
In \cite{cd1} and \cite{cd}, Cosgrave and Dilcher study the $1$-exceptional primes for $m = 3$ and $4$, and recently the author has shown that $p$ is $1$-exceptional for $m = 3$ if and only if $\lambda_p(\Q(\sqrt{-3})) > 1$, and $p$ is $1$-exceptional for $m = 4$ if and only if $\lambda_p(\Q(i)) > 1$ (see Theorem 1.1 in \cite{stokes}).  In the same paper, the author establishes a connection between the $1$-exceptional primes for $m = 3$ and $4$, and the Glaisher and Euler numbers respectively (see Corollaries 1.3 and 1.4 of \cite{stokes}).  Thus we have:

\begin{proposition}\label{thrm2}
Let $p \equiv 1 \Mod 3$, and consider $E/\F_p: y^2 = x^3 - 1$.  Then $p$ is $1$-exceptional for $m = 3$ if and only if $G_{p-1} \equiv 0 \Mod{p^2}$ if and only if  $\#E(\F_q) \equiv 0 \Mod{p^2}$, where $q = p^{p-1}$, and $\{G_n\}$ are the Glaisher numbers given by
\[
\frac{3/2}{e^x + e^{-x} + 1} = \sum_{n = 0}^{\infty} G_n \frac{x^n}{n!}.
\]
\end{proposition}

\begin{corollary}\label{cor1}
Let $p \equiv 1 \Mod 3$, and consider $E/\F_p: y^2 = x^3 - 1$.  If $p^2 = 3n^2 + 3n + 1$ for some $n \in \Z$, then $\#E(\F_q) \equiv 0 \Mod{p^2}$.
\end{corollary}

\begin{proof}
This follows from Proposition \ref{thrm2} and Corollary 6 in \cite{cd1}.
\end{proof}

\begin{proposition}\label{thrm3}
Let $p \equiv 1 \Mod 4$, and consider $E/\F_p: y^2 = x^3  + x$.  Then $p$ is $1$-exceptional for $m = 4$ if and only if $E_{p-1} \equiv 0 \Mod{p^2}$ if and only if  $\#E(\F_q) \equiv 0 \Mod{p^2}$, where $q = p^{p-1}$, and $\{E_n\}$ are the Euler numbers given by
\[
\frac{2}{e^x + e^{-x}} = \sum_{n = 0}^{\infty} E_n \frac{x^n}{n!}.
\]
\end{proposition}

Next, let $q = p^n$ for some $n \in \Z^+$, and denote $\psi_q : \F_q \to \{\pm 1\}$ to be the nontrivial character of order $2$.  If $E/\F_q$ is an elliptic curve given by $y^2 = f(x)$, then we have
\[
\# E(\F_q) = 1 + q + \sum_{x \in \F_q} \psi_q(f(x)).
\]
Therefore, by Theorem \ref{thrm1}:
\begin{proposition}\label{prop2}
Let $K$ be an imaginary quadratic field, and $p $ an odd prime that splits in $K$.  Suppose that $E$ is an elliptic curve with CM by $\mO_K$ with good reduction by the prime $\frak{P}$ of $K(j(E))$ above $p$.  Further, assume that the residue degree $f(\frak{P}/\frak{p}) \leq p-1$.  Denote $\tilde{E}: y^2 = f(x)$ to be the reduction of $E$ modulo $\frak{P}$.  Then 
\[
\lambda_p(K) > 1 \iff \sum_{x \in \F_q} \psi_q(f(x)) \equiv -1 \Mod{p^2}
\]
where $q = p^{p-1}$.
\end{proposition}

Note that in Proposition \ref{prop2}, the condition $f(\frak{P}/\frak{p}) \leq p-1$  ensures that $\tilde{E}$ is defined over $\F_q$.  This is a mild condition, since $f(\frak{P}/\frak{p}) \leq h_K$, and $p$ is usually large.  We will record one more consequence of Theorem \ref{thrm1}.

\begin{proposition}\label{prop1}
Let $p >3$ be a prime, and denote $\mathscr{K}_p$ to be the set of imaginary quadratic fields $K$ such that,
\begin{itemize}
\item[i.] $p$ splits in $K$
\item[ii.] The primes above $p$ are principal in $\mO_K$
\item[iii.] There is a elliptic curve that has CM by $\mO_K$ and good reduction at $\frak{P}$, which is a prime of the Hilbert class field of $K$ above $p$
\end{itemize}
Consider the finite set $\mathcal{E}_p$ of ordinary elliptic curves defined over $\F_p$.  If there is no $E \in \mathcal{E}_p$ such that $\# E(\F_{p^{p-1}}) \equiv 0 \Mod{p^2}$, then $\lambda_p(K) = 1$ for all $K \in \mathscr{K}_p$.
\end{proposition}

\begin{proof}
Let $p > 3$ be a prime and $K \in \mathscr{K}_p$.  Since $p$ is principal in $K$, we have that $p$ splits completely in the Hilbert class field of $K$.  Therefore, reducing an elliptic curve with CM by $\mO_K$ by a prime in the Hilbert class field above $p$ gives a curve in $\mathcal{E}_p$.  Now simply apply Theorem \ref{thrm1} to get the desired result.
\end{proof}

\begin{example}
Keep the notation from Proposition \ref{prop1}.  Then we can check using any standard computer algebra system that for $3< p < 50$, we have $\lambda_p(K) = 1$ for all $K \in \mathscr{K}_p$ when $p = 7, 17, 19, 29, 31, 37, 43$ and $47$.  It would be interesting to know for a fixed prime $p> 3$ whether or not the set $\mathscr{K}_p$ is infinite.
\end{example}

\section{Acknowledgments}

The author wishes to thank Nancy Childress for our many constructive discussions on this topic, and the referee for helpful comments and suggestions.

\end{document}